\documentclass{article}

\usepackage{amsmath,amsthm,amssymb,amsfonts}
\usepackage{verbatim}

\usepackage{stmaryrd} 
\usepackage{hyperref}

\usepackage[colorinlistoftodos]{todonotes}

\usepackage{tikz}
\usepackage{tkz-berge, tkz-graph}
\tikzset{vertex/.style={circle,draw,fill,inner sep=0pt,minimum size=1mm}}

\theoremstyle{plain}
\newtheorem{thm}{Theorem}
\newtheorem{lem}[thm]{Lemma}

\newtheorem{cor}[thm]{Corollary}
\newtheorem{conj}[thm]{Conjecture}

\theoremstyle{definition}
\newtheorem{definition}[thm]{Definition}
\newtheorem{exl}[thm]{Example}
\newtheorem{question}[thm]{Question}

\numberwithin{thm}{section}

\newcommand{\adj}{\leftrightarrow}
\newcommand{\adjeq}{\leftrightarroweq}

\newcommand{\pd}{\partial}

\def\Z{{\mathbb Z}}
\def\R{{\mathbb R}}

\title{A Borsuk-Ulam theorem for digital images}
\author{P.~Christopher Staecker}

\begin{document}

\bibliographystyle{plain}

\maketitle

\begin{abstract}
The Borsuk-Ulam theorem states that a continuous function $f:S^n \to \R^n$ has a point $x\in S^n$ with $f(x)=f(-x)$. We give an analogue of this theorem for digital images, which are modeled as discrete spaces of adjacent pixels equipped with $\Z^n$-valued functions. 

In particular, for a concrete two-dimensional rectangular digital image whose pixels all have an assigned ``brightness'' function, we prove that there must exist a pair of opposite boundary points whose brightnesses are approximately equal. This theorem applies generally to any integer-valued function on an abstract simple graph.

We also discuss generalizations to digital images of dimension 3 and higher. We give some partial results for higher dimensional images, and show a counter example which demonstrates that the full results obtained in lower dimensions cannot hold generally. 
\end{abstract}

\section{Introduction}
The Borsuk-Ulam is the following theorem in classical topology: If $f:S^n \to \R^n$ is continuous, then there is some point $x\in S^n$ such that $f(x)=f(-x)$, where $-x\in S^n$ is the antipodal point of $x$. The typical layman's statement of the Borsuk-Ulam theorem is that there is always a pair of opposite points on the surface of the Earth having the same temperature and barometric pressure. Since the theorem first appeared (proved by Borsuk) in the 1930s, many equivalent formulations, applications, alternate proofs, generalizations, and related ideas have been developed, and active work continues today. See the nice book \cite{mato03} by Matou\v{s}ek for a survey.

\begin{figure}
\[ \includegraphics{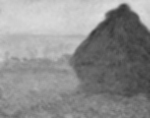} \]
\caption{Monet's \emph{Grainstack (Sunset)}. A Borsuk-Ulam Theorem would assert the existence of a pair of ``opposite'' pixels on the boundary having the same shade of gray.\label{haystackfig}}
\end{figure}

This paper aims to derive a similar theorem in the context of digital images which are defined by a discrete set of pixels. To get a feel for our aims, consider the grayscale digital image pictured in Figure \ref{haystackfig}, which measures $150\times 118$ pixels. The brightness of each pixel varies across the image, so we can consider the brightness as a function $f:B \to \Z$, where $B$ is the set of pixels in the image. The rectangular boundary $\pd B$ of the image is topologically analagous to $S^1$, and our brightness function naturally restricts to a function $f:\pd B \to \Z$. Then the classical Borsuk-Ulam theorem suggests that, subject to some continuity condition on $f$, there should be a pair of opposite points on the boundary having the same brightness. 

For digital images it turns out that we cannot in general expect opposite points with exactly the same brightness. We will revisit the image in Figure \ref{haystackfig} in Section \ref{exlsection} after we derive our results. 

Our approach uses the topological theory of digital images which has been developing since the 1980s. Several different models of ``digital topology'' have been proposed by various authors. We will use the graph-like theory of digital images and continuous functions based on adjacency relations. See Boxer \cite{boxe94} for a basic reference. We will review the basic ingredients here.

A \emph{digital image} is a pair $(X,\kappa)$, where $X$ is a set and $\kappa$ is a symmetric antireflexive relation on $X$, called the \emph{adjacency relation}. We sometimes refer elements of $X$ as \emph{pixels}, and the relation $\kappa$ tells us which pixels are adjacent to one another. When $a,b\in X$ and $a$ and $b$ are related by $\kappa$, we say $a$ and $b$ are \emph{$\kappa$-adjacent}, and we write $a \adj_\kappa b$. If the adjacency relation $\kappa$ is understood, we simply write $a\adj b$. If $a\adj b$ or $a=b$, we write $a\adjeq b$.

Typically it is assumed that $X\subset \Z^n$ for some $n$, and the adjacency relation $\kappa$ is a restriction of some globally defined adjacency relation on $\Z^n$. The most obvious applications occur in the case where $X$ is a finite subset of $\Z^2$ or $\Z^3$.

The topological theory of digital images has, to a large part, been characterized by taking ideas from classical topology and ``discretizing'' them. Typically $\R$ is replaced by $\Z$, and so on. 
Topology in $\R^n$ typically makes use of the standard topology on $\R^n$. 
Viewing $\Z$ as a digital image, it makes sense to use the following adjacency relation: $a,b\in \Z$ are called \emph{$c_1$-adjacent} if and only if $|a-b|=1$. This adjacency relation corresponds to connectivity in the standard topology of $\R$ (the notation ``$c_1$'' is explained below). 

Unfortunately there is no canonical ``standard adjacency'' to use in $\Z^n$ which corresponds naturally to the standard topology of $\R^n$. In the case of $\Z^2$, for example, at least two different adjacency relations seem reasonable: we can view $\Z^2$ as a rectangular lattice connected by the coordinate grid, so that each point is adjacent to 4 neighbors; or we can additionally allow diagonal adjacencies so that each point is adjacent to 8 neighbors. This is formalized in the following definition from \cite{han05} (though these adjacencies had been studied for many years earlier):

\begin{definition}
Let $k,n$ be positive integers with $k\le n$. Then define an adjacency relation $c_k$ on $\Z^n$ as follows: two points $x,y\in \Z^n$ are $c_k$-adjacent if their coordinates differ by at most 1 in at most $k$ positions, and are equal in all other positions.
\end{definition}

Thus $(\Z^2,c_1)$ is the two-dimensional integer lattice described above in which each point has 4 neighbors, and $(\Z^2,c_2)$ is the variation in which each point has 8 neighbors. The usual adjacency described above on $\Z$ is simply $c_1$. In $\Z^3$ the three relations $c_1,c_2,c_3$ differ with respect to which kinds of diagonal adjacencies are allowed: $c_1$ allows no diagonal adjacencies, $c_2$ allows ``face diagonals'' across unit squares, while $c_3$ additionally allows ``solid diagonals''. 

In this paper we focus on these $c_k$ relations. Almost all digital images in this paper will be subsets of $\Z^n$ for some $n$, with adjacency given by $c_k$ for some $k\le n$. Our results in Section \ref{dim1section} will hold for more general adjacency relations, and even digital images $X$ with no particular embedding in any $\Z^n$. Without restricting to subsets of $\Z^n$, the theory of digitial images resembles abstract graph theory (see \cite{hmps15}, which never assumes that the images are embedded in $\Z^n$).

Continuity of functions is a basic ingredient in a digital topological theory. The following definition is equivalent to one posed by Rosenfeld in \cite{rose86}:
\begin{definition}
Let $(X,\kappa)$ and $(Y,\lambda)$ be digital images, and let $f:X\to Y$ be a function. We say $f$ is \emph{$(\kappa,\lambda)$-continuous} (or simply \emph{continuous}, if the adjacencies are understood) if $a\adj_\kappa b$ implies $f(a)\adjeq_\lambda f(b)$. 
\end{definition}

We will also make use of the notion of connectedness. 
For $c,d\in \Z$  and $c<d$, let $[c,d]_\Z$ denote the set $\{c,c+1,\dots, d\}$. This set is called the \emph{digital interval} from $c$ to $d$.
Given two points $a,b\in (X,\kappa)$, a \emph{$\kappa$-path from $a$ to $b$} is a $(c_1,\kappa)$-continuous function $p: [0,n]_\Z \to X$ with $p(0)=a$ and $p(n)=b$. When the adjacency relation is understood, a $\kappa$-path is called simply a \emph{path}. 

A digital image $(X,\kappa)$ is \emph{$\kappa$-connected} (or simply \emph{connected}) when any two points of $X$ can be joined by a $\kappa$-path. A subset $S\subset X$ is a $\kappa$-connected component of $X$ when it includes all points which can be connected by a $\kappa$-path to its elements.

In many papers the concepts defined above are always prefixed by the word ``digital'' (e.g. ``digital path'', ``digitally continuous'', etc). This is usually unnecessary for us, but we will sometimes use the word for emphasis. 

The rest of this paper is organized as follows: in Section \ref{lipsection} we develop some necessary terminology and machinery for working with non-continuous functions. This material is used in Section \ref{dim1section} to obtain a Borsuk-Ulam theorem for $\Z$-valued functions on digital images. Our results in Section \ref{dim1section} generalize results by Boxer in \cite[Section 5]{boxe10}. In Section \ref{exlsection} we demonstrate a specific application to the Grainstack image. In Section \ref{higherdimsection} we give some partial results and state some open questions for higher dimensional images.

We would like to thank Lawrence Boxer for several comments and helpful references.

\section{Lipschitz constants for functions on digital images}\label{lipsection}
The proof of our Borsuk-Ulam theorem will require us to consider some functions which are not continuous. All functions we will consider do, however, have the following weaker property.

\begin{definition}
Let $(X,\kappa)$ and $(Y,\lambda)$ be digital images, and let $f:X\to Y$ be a function. We say $f$ has \emph{[$(\kappa,\lambda)$]-Lipschitz constant} $m\in \Z$ if $a\adj_\kappa b$ implies that there is a $\lambda$-path of length $m$ from $f(a)$ to $f(b)$.
\end{definition}
The terminology is motivated by the classical definiton of the Lipschitz constant, which would require that $\frac{d(f(a),f(b))}{d(a,b)} \le m$. In our case to avoid assuming a particular metric on $X$ we only refer to $\kappa$-adjacent points in which case it makes sense to interpret $d(a,b)=1$, and the distance from $f(a)$ to $f(b)$ is interpreted in terms of the length of a $\lambda$-path connecting them.

This notion corresponds to a Lipschitz condition given by Rosenfeld \cite[Section 5]{rose86}, which assumes that $X$ and $Y$ are metric spaces. Rosenfeld does not explore the idea very much, and we will require a bit more development. 

By definition a function is digitally continuous if and only if it has Lipschitz constant 1. Thus, unlike in classical analysis, the existence of a Lipschitz constant does not imply continuity. In fact in typical cases, any function (continuous or not) has a Lipschitz constant:
\begin{exl}
Let $f:(X,\kappa) \to (Y,\lambda)$ be any function (not necessarily continuous) such that $f(X)$ is a finite set contained in a single connected component of $Y$. In this case any two points of $f(X)$ are joined by a path in $Y$. Since $f(X)$ is finite, there must be some maximum length required to realize all of these paths. Let $m$ be this maximum, and then $f$ has Lipshitz constant $m$.
\end{exl}

For integer-valued functions $f:(X,\kappa) \to (\Z,c_1)$, it is easy to see that $f$ has Lipschitz constant $m$ if and only if $a\adjeq_\kappa b$ implies $|f(a)-f(b)|\le m$. 
Lipschitz constants for such functions behave predictably with respect to addition and scalar multiplication:
\begin{thm}\label{lipaddition}
Let $f,g:(X,\kappa) \to (\Z,c_1)$ have Lipschitz constants $m$ and $l$ respectively, and let $c\in \Z$. Then:
\begin{itemize}
\item $f+g:X\to \Z$ has Lipschitz constant $m+l$.
\item $cf:X\to \Z$ has Lipschitz constant $|c|m$. \qed
\end{itemize}
\end{thm}
The proofs are routine. Since continuous functions have Lipschitz constant 1, we have:
\begin{cor}\label{ctsaddition}
If $f,g:X\to \Z$ are continuous, then $f+g$ and $f-g$ each have Lipschitz constant 2. \qed
\end{cor}

Lipschitz constants also behave predictably with respect to compositions. Again the proof is routine.
\begin{thm}
If $g:X\to Y$ has Lipschitz constant $l$ and $f:Y \to Z$ has Lipschitz constant $m$, then $f\circ g:X\to Z$ has Lipschitz constant $ml$. \qed
\end{thm}

Since a continuous function has Lipschitz constant 1, the above gives:
\begin{cor}\label{lipcomposition}
If $g:X\to Y$ is continuous and $f:Y \to Z$ has Lipschitz constant $m$, then $f\circ g:X\to Z$ has Lipschitz constant $m$. \qed
\end{cor}

$(\kappa,\lambda)$-Lipschitz constants can be equivalently formulated in terms of continuity with respect to a different adjacency. For a digital image $(Y,\lambda)$, let $\lambda^k$ be the adjacency relation defined by: $a$ is $\lambda^k$-adjacent to $b$ if and only if there is a $\lambda$-path of length $k$ from $a$ to $b$. This exponentiation operation has the following simple properties:
\begin{itemize}
\item $\lambda^1 = \lambda$
\item $(\lambda^m)^k = \lambda^{mk}$
\end{itemize}
Most useful for us will be the following relation between powers of $\lambda$ and Lipschitz constants. Again we omit the routine proof.
\begin{thm}\label{powerlip}
Let $(X,\kappa)$ and $(Y,\lambda)$ be digital images, and let $f:X\to Y$ be any function. Then $f$ has $(\kappa,\lambda)$-Lipschitz constant $m$ if and only if $f$ is $(\kappa, \lambda^m)$-continuous. \qed
\end{thm}

\section{A Borsuk-Ulam theorem for $\Z$-valued functions}\label{dim1section}
In dimension 1, the Borsuk-Ulam Theorem says that any continuous map $f:S^1 \to \R$ has a point $x$ with $f(x)=f(-x)$. The classical proof takes the map $h(x) = f(x)-f(-x)$, and applies the intermediate value theorem to show that $h$ must have the value 0 for some $x$. 

Our adaptation of this proof will require a digital analog of the intermediate value theorem which was proved by Rosenfeld in \cite[Theorem 3.1]{rose86}. After completing our work we learned of the material by Boxer in \cite{boxe10}, which gives a special case of our result in this section. Boxer uses essentially the same ideas, but focuses only on continuous functions (rather than functions with Lipschitz constant $m>1$), and focuses only on the case where the domain is a ``digital simple closed curve''. We are grateful to L. Boxer for bringing this material to our attention. 

The difficulty in adapting the classical proof  of the Borsuk-Ulam theorem to digital images is that, even when $f$ is continuous, the function $f(x)-f(-x)$ may not be continuous, and thus Rosenfeld's intermediate value theorem does not apply. By Corollary \ref{ctsaddition}, however, $f(x)-f(-x)$ will have Lipschitz constant 2. Thus we will need to generalize Rosenfeld's theorem for functions with Lipschitz constant greater than 1. The theorem below is a direct generalization of the digital intermediate value theorem: when $f$ is continuous we have $m=1$ and the conclusion reads $f(z)=c$.
\begin{thm}\label{ivt}
Let $X$ be connected and $f:(X,\kappa)\to (\Z,c_1)$ have Lipschitz constant $m>0$, and let $x,y \in X$ and $c\in \Z$ with $f(x)\le c \le f(y)$. Then there is a point $z \in X$ with $|f(z) - c| < m$.
\end{thm}

\begin{proof}
If $c=f(x)$ or $c=f(y)$ then the conclusion follows trivially. Thus we will assume that $f(x) < c < f(y)$.

Let $p:[0,t]_\Z \to X$ be a path from $x$ to $y$. Since $f(x)< c< f(y)$ and $f(x)\neq f(y)$, some values of $f\circ p$ are greater than or equal to $c$ and some are less than $c$. Thus is some $s$ such that $f(p(s-1)) < c \le f(p(s))$. Let $z=p(s)$, and we have:
\[ |f(z)-c| = |f(p(s))-c| = f(p(s))-c < f(p(s)) - f(p(s-1)) \le m. \]
By Corollary \ref{lipcomposition}, $f \circ p:[0,t]_\Z \to \Z$ has Lipschitz constant $m$, which gives the last inequality above. We have shown that $|f(z)-c| < m$ as desired.
\end{proof}

Recall that the standard dimension 1 Borsuk-Ulam theorem concerns functions $f:S^1\to \R$. For our digital analogue we will replace $\R$ by $(\Z,c_1)$, and it seems most natural to use a simple cycle of points to replace $S^1$. Let $C_n$ be a digital image given by $C_n = \{c_0,\dots, c_{n-1}\}$ with the adjacency relation $\gamma$ given by $c_i \adj_{\gamma} c_j$ if and only if $|i-j| = 1 \pmod n$. For a point $c_i \in C_n$, define the ``antipodal point'' $-c_i$ by $-c_i = c_{i+k}$, where $k=\lfloor n/2 \rfloor$.

The natural statement of the Borsuk-Ulam theorem in our context would then be the following: If $f:(C_n,\gamma) \to (\Z,c_1)$ is continuous, then there is a point $x\in C_n$ with $f(x)=f(-x)$. This is not true, as the following counterexample demonstrates: (a similar example appears in \cite[Figure 7]{boxe10})
\begin{exl}
Let $f:C_8 \to \Z$ be the map given by: 
\[ (f(c_0), \dots, f(c_7)) = (0,0,1,2,2,2,2,1). \]
Then $f$ is $(\gamma, c_1)$-continuous, and $f(-c_i)$ have the following values:
\[ (f(-c_0), \dots, f(-c_7)) = (2,2,2,1,0,0,1,2), \]
and we can see that $f(c_i) \neq f(-c_i)$ for all $i$.
\end{exl}

The classical Borsuk-Ulam theorem fails to hold for the above example because the function $f(x)-f(-x)$ fails to be continuous, and thus the intermediate value theorem does not hold. We will need to weaken the conclusion in our Borsuk-Ulam theorem in light of the slightly weaker conclusion to our intermediate value theorem from Theorem \ref{ivt}.

It turns out, as we will see below, that the domain space need not be a cycle $C_n$, but needs only be a digital image with some continuous ``antipodality'' function. A function $\tau:X\to X$ is an \emph{involution} when it is continuous and $\tau(\tau(x)) = x$ for all $x\in X$. If $\tau$ has no fixed points, it is a \emph{free involution}.
We say that a digital image $(X,\kappa)$ is \emph{$\tau$-involutive} if it admits a $\kappa$-continuous free involution $\tau:X\to X$. By analogy with $S^n$, we will write $\tau(x) = -x$ when no confusion will arise, and simply call the image \emph{involutive}. 

We are now ready to prove our digital Borsuk-Ulam for $\Z$-valued functions. We follow exactly the classical proof. (See \cite[Theorem 5.1]{boxe10} for a slightly weaker form of this result.)
\begin{thm}\label{BUdim1}
Let $(X,\kappa)$ be a connected involutive digital image, and let $f:(X,\kappa) \to (\Z,c_1)$ have Lipschitz constant $m > 0$. Then there is a point $x\in X$ such that $|f(x) - f(-x)| < 2m$. 
\end{thm}
\begin{proof}
Since $f(x)$ has Lipschitz constant $m$, it is easy to see that the function $f(-x)$ also has Lipschitz constant $m$. Let $h:X \to \Z$ be defined by $h(x) = f(x) - f(-x)$. By Lemma \ref{lipaddition}, $h$ has Lipschitz constant $2m$. To obtain a contradiction, assume that $|f(x) - f(-x)| \ge 2m$ for all $x$, that is, $|h(x)|\ge 2m$ for all $x$. 

Since the function $x \mapsto -x$ is an involution, we have $h(-x)=-h(x)$ for all $x$. Since $h$ never takes the value 0, this means that $h(x)$ is positive for some $x$ and negative for others. Thus there is some point $a\in X$ such that $h(a) > 0$ and thus $h(-a) = -h(a)< 0$. Thus we have $h(-a) < 0 < h(a)$, and by Theorem \ref{ivt} there is some $z\in X$ with $|h(z)|<2m$. This contradicts our assumption that $|h(x)|\ge 2m$ for all $x$.
\end{proof}

Equivalently we can state the theorem in terms of powers of the $c_1$ adjacency on $\Z$. By Theorem \ref{powerlip} we have:
\begin{cor}\label{BUdim1cor}
Let $(X,\kappa)$ be a connected involutive digital image, and let $f:X \to \Z$ be $(\kappa, c_1^m)$-continuous for $m\ge 0$. Then there is a point $x\in X$ such that $f(x) \adjeq_{c_1^{2m-1}} f(-x)$. 
\end{cor}

In the case where $f$ is continuous we can replace $m$ above by 1, and we obtain:
\begin{cor}
Let $(X,\kappa)$ be a connected involutive digital image, and let $f:(X,\kappa) \to (\Z,c_1)$ be continuous. Then there is a point $x\in X$ with $f(x)\adjeq_{c_1} f(-x)$. 
\end{cor}

We conclude the section with one more restatement of Theorem \ref{BUdim1}. Since we make no assumptions on $X$ as embedded in $\Z^n$ with any particular adjacency, we can state the result in the context of pure graph theory. In this context, the digital image $(X,\kappa)$ is viewed as a simple graph (undirected, no looped or multiple edges) with vertex set $X$ and edges corresponding to $\kappa$-adjacencies. The existence of a continuous free involution, in graph theoretic terms, is equivalent to $X$ having an automorphism which is a free involution.

\begin{cor}
Let $X$ be a simple graph with an automorphism $\tau$ which is a free involution, and let $f:X\to \Z$ be any function on the vertex set, and assume that $|f(a)-f(b)| \le m$ whenever $x$ and $y$ are connected by an edge. Then there is some point $x\in X$ with $|f(x)-f(\tau(x))| < 2m$. 
\end{cor}


\section{A specific example using a grayscale digital image}\label{exlsection}
A typical bitmapped digital image on a computer (for example in jpg or png format) is stored as a rectangular array of color values. For simplicity we will consider grayscale images, which can be stored as a two-dimensional array of integers, where the integer represents the brightness of each pixel. Consider Figure \ref{haystackfig}, which is a low-resolution grayscale image of Monet's \emph{Grainstack (Sunset)}. At this resolution and color-depth, this image is represented as a $150\times 118$ array of integers in the range $[0,255]_\Z$, where 0 indicates a pure black pixel, and 255 indicates pure white. We adopt the standard computer-graphics convention that pixel $(0,0)$ is in the top-left corner of the displayed image.

In our abstract model, the Grainstack image is the 2-box $B = [0,149]_\Z\times [0,117]_\Z$, equipped with a ``brightness'' function $f:B \to [0,255]_\Z$. Let $X$ be the ``boundary'' of $B$, the perimeter cycle of 536 points. The natural free involution on $X$ is the function $\tau(a,b) = (149 - a, 117 - b)$ which exchanges ``opposite'' points. 

For an image of this size, it is easy to check continuity of $f$ by comparing the brightness of adjacent pixels. The function $f$ will be $(c_k,c_1)$-continuous if all $c_k$-adjacent pixels have brightnesses differing by at most 1. 

Few typical grayscale images will be $(c_k,c_1)$ continuous, since this precludes any sharply-defined features in the image. The $(c_k,c_1)$-Lipschitz constant (typically not equal to 1) is a measure of how sharply the brightness can change between any adjacent pixels. This can easily be computed in our example image.

In our specific example, the restriction of $f$ to $X$ has $(c_2, c_1)$-Lipschitz constant 23. The greatest change in brightness between adjacent pixels on the boundary occurs on the right edge of the image where the haystack meets the sky: between pixels (149,29) and (149,30) the brightness changes by 23.

Theorem \ref{BUdim1} applies, and states that there is a point $x\in X$ with $|f(x) - f(-x)| < 2\cdot 23 = 46$. That is, there is some pair of opposite points whose brightness differ by less than 46. This is also easy to verify for an image of this size, and in fact the antipodal points with the closest brightness values are $(0,87)$ and $(149,30)$, which have brightness differing by 13. These are marked on Figure \ref{haystackmarked}.

\begin{figure}
\[ \includegraphics[scale=.5]{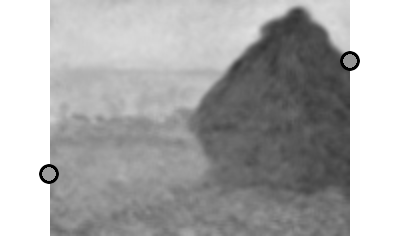} \]
\caption{The image from Figure \ref{haystackfig} with antipodal pixels of closest brightness indicated\label{haystackmarked}}
\end{figure}

Informally speaking, in the case of typical grayscale images, our result states that the the largest possible difference in brightness between antipodal boundary points is two times the largest change in brightness between adjacent boundary points.

\section{Higher dimensions}\label{higherdimsection}
A higher dimensional Borsuk-Ulam theorem would consider maps from certain involutive digital images into $\Z^n$ with $n>1$. Two factors will complicate matters in this more general setting. For $n>1$, we have several competing options for the adjacency to use on the codomain space $\Z^n$. It is also not obvious what the assumptions on the digital image $X$ should be, in order that it resemble the sphere $S^n$. 

Several authors have considered ``sphere-like'' digital images. We adapt a formulation from \cite{boxe06}, where Boxer defines an $n$-sphere as $S_n \subset \Z^{n+1}$ given by $S_n = [-1,1]^{n+1}_\Z - \{\mathbf 0\}$, where $\mathbf 0$ is the origin in $\Z^{n+1}$. We will consider slightly more general digital images: our results hold for any image which is the boundary of an $n$-box:
\begin{definition}
A set $B \subset \Z^n$ is called a \emph{[digital] $n$-box} if it has the form:
\[ B = [a_1,b_1]_\Z \times \dots \times [a_n,b_n]_\Z, \] 
where $a_i,b_i \in \Z$ with $a_i>b_i$ for each $i$. If $b_i = a_i+1$ for each $i$, we say $B$ is a \emph{unit} $n$-box.

For an $n$-box $B$, let $\pd B$ denote the \emph{boundary} of $B$, defined by:
\[ \pd B = \bigcup_{i=1}^n (\{ x \in B : x_i = a_i\} \cup \{x \in B : x_i = b_i\}). \]
The $n$-box $B$ carries a canonical involution $\tau:B \to B$ given by:
\[ \tau(x_1,\dots x_n) = ( a_1 + b_1 - x_1, \dots, a_n + b_n - x_n )\]
This is an ``antipodal'' reflection of $B$, which restricts to a free involution on $\pd B$. As usual, we will denote $\tau(x) = -x$ (note that this will not equal the usual additive inverse of $x\in \Z^n$).
\end{definition}

Boxer's $S_n$ is the boundary of a $(n+1)$-box $S_n = \pd([-1,1]_Z^{n+1})$. 

Our expected formulation of the Borsuk-Ulam theorem, then, is the following: If $B$ is an $n$-box and $f:\pd B \to \Z^{n-1}$ is continuous, then there is some point $x$ where $f(x) \adjeq f(-x)$.

Our main approach will take a digitally continuous function $f:\pd B \to \Z^n$ and find the desired point $x$ by replacing $f$ by a continuous map on $S^{n-1} \to \R^n$ and applying the classical Borsuk-Ulam theorem. 

For any subset $S\subset \Z^n$, let $|S|\subset \R^n$ be its convex hull. In particular if $B\subset \Z^n$ is an $n$-box, we have 
\[ |B| = [a_1,b_1] \times \dots \times [a_n,b_n], \]
and $|B|$ is homeomorphic to the closed $n$-ball in $\R^n$. 

Let $\pd |B|$ be the topological boundary of $|B|$, or equivalently,
\[ \pd |B| = \bigcup_{i=1}^n (\{ x \in |B| : x_i = a_i\} \cup \{x \in |B| : x_i = b_i\}). \]

Let $\mathcal T$ be any triangulation of $\R^n$ whose vertex set is $\Z^n$. Any function $f:\Z^n \to \Z^m$ induces a piecewise linear real-valued map $F:\R^n \to \R^m$ by extending $f$ linearly along the simplices of $\mathcal T$. This $F$ we call the \emph{linearization of $f$ on $\mathcal T$}. Such linearizations respect simplices in the following sense:

\begin{lem}\label{linearizationlem}
Let $F:\R^n\to \R^m$ be the linearization of $f:\Z^n \to \Z^m$ on some triangulation $\mathcal T$. Let $\sigma$ be the vertex set of a simplex of $\mathcal T$. Then $F(|\sigma|) = |f(\sigma)|$. 
\end{lem}
\begin{proof}
Let $\sigma = \{x_0, \dots, x_k\}$. The set $F(|\sigma|)$ is $F$ applied to the convex hull of $x_0, \dots, x_k$, and since $F$ is linear on the interior of $|\sigma|$, this is the convex hull of $F(x_0),\dots, F(x_k)$. Since $F$ and $f$ agree on the integer lattice and $x_i\in \Z^n$ for each $i$, we have that $F(|\sigma|)$ is the convex hull of $f(x_0), \dots, f(x_k)$, which equals $|f(\sigma)|$. 
\end{proof}

In a digital image $(X,\kappa)$, a simplex (or $\kappa$-simplex) (see \cite{ako08}) is a set of points of $X$ which are mutually $\kappa$-adjacent. 

Each of our higher-dimensional Borsuk-Ulam theorems rely on the following lemma about digital simplices.
\begin{lem}\label{simplexlem}
Let $B\subset \Z^{n}$ be an $n$-box, and let $f:\pd B \to \Z^{n-1}$ be any function. Then there is a $c_n$-simplex $\sigma \subset \pd B$  such that $|f(\sigma)| \cap |f(-\sigma)| \neq \emptyset$. 
\end{lem}
\begin{proof}
Let $\mathcal T$ be any antipodally symmetric triangulation of $\pd |B|$ whose 1-skeleton is a subset of the edges given by $c_n$. (``Antipodally symmetric'' means that if $\sigma$ is a simplex of $\mathcal T$ then $-\sigma$ is a simplex of $\mathcal T$, where the negation denotes the antipodality involution of $\pd B$.)

Let $F:\pd |B| \to \R^{n-1}$ be the linear extension onto $\mathcal T$ of $f:\pd B \to \Z^{n-1}$. Since $\pd |B|$ is homeomorphic to $S^n$ the classical Borsuk-Ulam theorem gives a point $z\in \pd |B|$ with $F(z)=F(-z)$. 

Let $\sigma \in \mathcal T$ be (the vertex set of) the simplex such that $z\in |\sigma|$. Then since $F(z)=F(-z)$ we must have $F(|\sigma|) \cap F(|-\sigma|) \neq \emptyset$. By Lemma \ref{linearizationlem} this means $|f(\sigma)| \cap |f(-\sigma)| \neq \emptyset$ as desired.
\end{proof}

Our first higher-dimensional Borsuk-Ulam theorem applies to the case where $c_1$-adjacency is used in the codomain of $f$.

\begin{thm}\label{BUcnc1}
Let $B\subset \Z^{n}$ be an $n$-box, and let $f:\pd B \to \Z^{n-1}$ be $(c_n,c_1)$-continuous. Then there is some point $x\in \pd B$ such that $f(x) \adjeq_{c_1} f(-x)$. 
\end{thm}
\begin{proof}
This proof is made possible by the rigid structure of simplices in $c_1$. Any $c_1$-simplex consists simply of a pair of points, and its convex hull is a line segment. Two $c_1$-simplexes whose convex hulls intersect must share a vertex.

By Lemma \ref{simplexlem} there is a $c_n$-simplex $\sigma$ with $|f(\sigma)| \cap |f(-\sigma)| \neq \emptyset$, and since both $f(\sigma)$ and $f(-\sigma)$ are $c_1$-simplices with intersecting convex hulls, $f(\sigma) \cap f(-\sigma)$ contains a vertex $y \in \Z^{n-1}$. 

Let $x\in \sigma$ be some point with $f(x) = y$, and then since $y\in f(-\sigma)$ and $f(-x) \in f(-\sigma)$ and $f(-\sigma)$ is a $c_1$-simplex we have $f(x) \adjeq_{c_1} f(-x)$ as desired.
\end{proof}

The above proof makes use of very specific properties of $c_1$ adjacency, and seems difficult to generalize to other adjacencies in a straightforward way. In the case $n=2$ the above corresponds to the $m=1$ case of Corollary \ref{BUdim1cor}. As we saw in Section \ref{exlsection}, typical digital images will not be $c_1$ continuous so it is important to obtain a result for $c_1^m$ adjacency in place of $c_1$. In fact we can generalize to higher values of $m$. We require a special construction for $c_1^m$-simplices in $\Z^n$.

Let $\sigma$ be a $c_1^m$-simplex in $\Z^n$, and let $p_\sigma \in \Z^n$ be the point whose coordinate in each position $i$ is the minimum of the $i$-th coordinates of the points of $\sigma$. Then let $T(\sigma)$ be the set:
\[ T(\sigma) = \{ p_\sigma + x \in \Z^n : \text{ all coordinates of $x$ are nonnegative and }s(x) \le m\}, \]
where $s(x) \in \Z$ is the sum of the coordinates of $x$. 
The points of $T(p)$ form a $c_1^m$-simplex. We have $\sigma \subset T(\sigma)$, and so $|\sigma| \subset |T(\sigma)|$. Also, given any two $c_1^m$-simplices $\sigma$ and $\rho$, if $|\sigma| \cap |\rho| \neq \emptyset$ then $|T(\sigma)|\cap |T(\rho)|$ is nonempty and contains at least one point of $\Z^n$.

We use these sets $T(\sigma)$ to obtain a Borsuk-Ulam theorem when $c_1^m$-adjacency is used in the codomain.
\begin{thm}\label{BUcnc1m}
Let $B\subset \Z^{n}$ be an $n$-box, and let $f:\pd B \to \Z^{n-1}$ be $(c_n,c_1^m)$-continuous. Then there is some point $x\in \pd B$ such that $f(x) \adjeq_{c_1^{2m}} f(-x)$. 
\end{thm}
\begin{proof}
Again we begin with Lemma \ref{simplexlem} to obtain a $c_n$-simplex $\sigma$ with $|f(\sigma)| \cap |f(-\sigma)| \neq \emptyset$. Since $f(\sigma)$ and $f(-\sigma)$ are $c_1^m$-simplices, the set $|T(f(\sigma))| \cap |T(f(-\sigma))|$ is nonempty and contains some integer point $y\in \Z^n$. 

Let $x \in \sigma$. Then $f(x) \in f(\sigma)$, and so $f(x) \adjeq_{c_1^m} y$ since $y$ and $f(x)$ are both in $T(f(\sigma))$ which is a $c_1^m$-simplex. Similarly we have $f(-x) \in f(-\sigma)$ and so $f(-x) \adjeq_{c_1^m} y$. Concatenating the $c_1$-paths from $f(x)$ to $y$ and from $y$ to $f(-x)$ gives $f(x) \adjeq_{c_1^{2m}} f(-x)$ as desired.
\end{proof}

Next we prove a similar result to Theorem \ref{BUcnc1} using the adjacency $c_{n-1}$ in place of $c_1$ in the codomain space. We will require the following lemma.
\begin{lem}\label{cnlemma}
Let $\sigma, \rho \subset \Z^n$ be $c_n$-simplices. If $|\sigma| \cap |\rho|\neq \emptyset$, then there is a vertex of $\sigma$ which is $c_n$-adjacent to all vertices of $\rho$. 
\end{lem}
\begin{proof}
Since $\sigma$ and $\rho$ are $c_n$-simplices, each is a subset of a unit $n$-box. Let $B$ and $C$ be unit $n$-boxes containing $\sigma$ and $\rho$, respectively. Then we have $|\sigma| \subset |B|$ and $|\rho| \subset |C|$, and so $|\sigma| \cap |\rho| \subset |B| \cap |C|$, and thus $|B| \cap |C| \neq \emptyset$. Since $B$ and $C$ are unit boxes, their intersection $|B|\cap|C|$ is a ``face'' of a box, which by construction must contain a vertex from each of $\sigma$ and $\rho$. Thus there is some vertex $x\in \sigma$ with $x\in B \cap C$. Since $\sigma \subset B$, every vertex of $C$ is $c_n$-adjacent to any vertex of $\rho$. Thus $x$ is $c_n$-adjacent to every vertex of $\rho$ as desired.
\end{proof}

The lemma gives another higher-dimensional Borsuk-Ulam theorem:

\begin{thm}\label{BUcncn-1}
Let $B\subset \Z^{n}$ be an $n$-box, and let $f:\pd B \to \Z^{n-1}$ be $(c_n,c_{n-1})$-continuous. Then there is some point $x\in \pd B$ such that $f(x) \adjeq_{c_{n-1}} f(-x)$. 
\end{thm}
\begin{proof}
We use a similar argument to that used for Theorem \ref{BUcnc1}, but this time we use Lemma \ref{cnlemma} in place of the specific arguments concerning $c_1$. 

Again by Lemma \ref{simplexlem} there is a $c_n$-simplex $\sigma$ with $|f(\sigma)| \cap |f(\sigma)| \neq \emptyset$. Then since $f$ is $(c_n,c_{n-1})$-continuous, $f(\sigma)$ and $f(-\sigma)$ are $c_{n-1}$-simplices of $\Z^{n-1}$, and so by Lemma \ref{cnlemma} there is a vertex $y\in f(\sigma)$ which is $c_{n-1}$-adjacent to all vertices of $f(-\sigma)$. Let $x\in \sigma$ be some point with $f(x) = y$, and then we have $f(x) \adjeq_{c_{n-1}} f(-x)$. 
\end{proof}

The proofs used for Theorems \ref{BUcnc1} and \ref{BUcncn-1} generalize directly to any adjacency relation satisfying the following condition: we say an adjacency relation $\lambda$ on $\Z^n$ is \emph{regular} when, if $\sigma$ and $\rho$ are $\lambda$-simplices and $|\sigma| \cap |\rho| \neq \emptyset$, then there is some vertex $x\in \sigma$ which is $\lambda$-adjacent to all points of $\rho$. Lemma \ref{cnlemma} states that $c_n$ is regular as an adjacency on $\Z^n$, and the arguments used in the proof of Theorem \ref{BUcnc1} show that $c_1$ is regular. 

It seems likely that all adjacencies $c_k$ on $\Z^n$ with $k \le n$ should be regular, but we have been unable to prove this, so we state it as a conjecture:
\begin{conj}
The adjacency relation $c_k$ on $\Z^n$ with $k\le n$ is regular.
\end{conj}

Subject to the conjecture above we can expect a nice Borsuk-Ulam result for $(c_n,c_k)$-continuous functions on $n$-boxes with $k<n$. It is natural to ask if this type of result can hold for $(c_l,c_k)$-continuous functions with both $l<n$ and $k<n$. The adjacency conclusion from Theorems \ref{BUcnc1} and \ref{BUcncn-1} will not hold in this more general case. The following example shows a $\Z^2$-valued function on the boundary of a $3$-box which is $(c_1,c_1)$-continuous but has no point $x$ with $f(x) \adjeq_{c_1} f(-x)$. 

\begin{figure}
\[ 
\begin{tikzpicture}[scale=.5]
	\filldraw[fill=white, xshift=-2cm,yshift=-2cm]
		(45:1.2) \foreach \x in {135,225,315,45} { -- (\x:1.2) };
	\filldraw[fill=white, xshift=-2cm,yshift=0cm]
		(45:1.2) \foreach \x in {135,225,315,45} { -- (\x:1.2) };
	\filldraw[fill=white, xshift=-2cm,yshift=2cm]
		(45:1.2) \foreach \x in {135,225,315,45} { -- (\x:1.2) };
	\filldraw[fill=white, xshift=0cm,yshift=-2cm]
		(45:1.2) \foreach \x in {135,225,315,45} { -- (\x:1.2) };
	\filldraw[fill=white, xshift=0cm,yshift=0cm]
		(45:1.2) \foreach \x in {135,225,315,45} { -- (\x:1.2) };
	\filldraw[fill=white, xshift=0cm,yshift=2cm]
		(45:1.2) \foreach \x in {135,225,315,45} { -- (\x:1.2) };
	\filldraw[fill=white, xshift=2cm,yshift=-2cm]
		(45:1.2) \foreach \x in {135,225,315,45} { -- (\x:1.2) };
	\filldraw[fill=white, xshift=2cm,yshift=0cm]
		(45:1.2) \foreach \x in {135,225,315,45} { -- (\x:1.2) };
	\filldraw[fill=white, xshift=2cm,yshift=2cm]
		(45:1.2) \foreach \x in {135,225,315,45} { -- (\x:1.2) };
	\node () at (-2cm,-2cm) {$(0,1)$};
	\node () at (-2cm,0cm) {$(0,0)$};
	\node () at (-2cm,2cm) {$(0,0)$};
	\node () at (0cm,-2cm) {$(0,0)$};
	\node () at (0cm,0cm) {$(0,0)$};
	\node () at (0cm,2cm) {$(0,0)$};
	\node () at (2cm,-2cm) {$(0,0)$};
	\node () at (2cm,0cm) {$(0,0)$};
	\node () at (2cm,2cm) {$(0,0)$};
\end{tikzpicture}
\qquad
\begin{tikzpicture}[scale=.5]
	\filldraw[fill=white, xshift=-2cm,yshift=-2cm]
		(45:1.2) \foreach \x in {135,225,315,45} { -- (\x:1.2) };
	\filldraw[fill=white, xshift=-2cm,yshift=2cm]
		(45:1.2) \foreach \x in {135,225,315,45} { -- (\x:1.2) };
	\filldraw[fill=white, xshift=0cm,yshift=-2cm]
		(45:1.2) \foreach \x in {135,225,315,45} { -- (\x:1.2) };
	\filldraw[fill=white, xshift=0cm,yshift=2cm]
		(45:1.2) \foreach \x in {135,225,315,45} { -- (\x:1.2) };
	\filldraw[fill=white, xshift=2cm,yshift=-2cm]
		(45:1.2) \foreach \x in {135,225,315,45} { -- (\x:1.2) };
	\filldraw[fill=white, xshift=2cm,yshift=2cm]
		(45:1.2) \foreach \x in {135,225,315,45} { -- (\x:1.2) };
	\filldraw[fill=white, xshift=-2cm,yshift=0cm]
		(45:1.2) \foreach \x in {135,225,315,45} { -- (\x:1.2) };
	\filldraw[fill=white, xshift=2cm,yshift=0cm]
		(45:1.2) \foreach \x in {135,225,315,45} { -- (\x:1.2) };
	\node () at (-2cm,-2cm) {$(1,1)$};
	\node () at (-2cm,2cm) {$(0,1)$};
	\node () at (0cm,-2cm) {$(1,0)$};
	\node () at (0cm,2cm) {$(0,1)$};
	\node () at (2cm,-2cm) {$(1,0)$};
	\node () at (2cm,2cm) {$(0,0)$};
	\node () at (-2cm,0cm) {$(0,1)$};
	\node () at (2cm,0cm) {$(1,0)$};
\end{tikzpicture}
\qquad
\begin{tikzpicture}[scale=.5]
	\filldraw[fill=white, xshift=-2cm,yshift=-2cm]
		(45:1.2) \foreach \x in {135,225,315,45} { -- (\x:1.2) };
	\filldraw[fill=white, xshift=-2cm,yshift=0cm]
		(45:1.2) \foreach \x in {135,225,315,45} { -- (\x:1.2) };
	\filldraw[fill=white, xshift=-2cm,yshift=2cm]
		(45:1.2) \foreach \x in {135,225,315,45} { -- (\x:1.2) };
	\filldraw[fill=white, xshift=0cm,yshift=-2cm]
		(45:1.2) \foreach \x in {135,225,315,45} { -- (\x:1.2) };
	\filldraw[fill=white, xshift=0cm,yshift=0cm]
		(45:1.2) \foreach \x in {135,225,315,45} { -- (\x:1.2) };
	\filldraw[fill=white, xshift=0cm,yshift=2cm]
		(45:1.2) \foreach \x in {135,225,315,45} { -- (\x:1.2) };
	\filldraw[fill=white, xshift=2cm,yshift=-2cm]
		(45:1.2) \foreach \x in {135,225,315,45} { -- (\x:1.2) };
	\filldraw[fill=white, xshift=2cm,yshift=0cm]
		(45:1.2) \foreach \x in {135,225,315,45} { -- (\x:1.2) };
	\filldraw[fill=white, xshift=2cm,yshift=2cm]
		(45:1.2) \foreach \x in {135,225,315,45} { -- (\x:1.2) };
	\node () at (-2cm,-2cm) {$(1,1)$};
	\node () at (-2cm,0cm) {$(1,1)$};
	\node () at (-2cm,2cm) {$(1,1)$};
	\node () at (0cm,-2cm) {$(1,1)$};
	\node () at (0cm,0cm) {$(1,1)$};
	\node () at (0cm,2cm) {$(1,1)$};
	\node () at (2cm,-2cm) {$(1,1)$};
	\node () at (2cm,0cm) {$(1,1)$};
	\node () at (2cm,2cm) {$(1,0)$};
\end{tikzpicture}
\]
\caption{A $(c_1,c_1)$-continuous function $f:\pd([-1,1]^3_\Z) \to \Z^2$ without any point $f(x)\adjeq f(-x)$. The printed coordinates show the values of $f$ at each pixel. At left is the ``bottom layer'' $[-1,1]_\Z^2 \times \{-1\}$, at right is the ``top layer'' $[-1,1]_\Z^2 \times \{1\}$, and in the center is the ``middle layer'' $[-1,1]_\Z^2 \times \{0\}$. \label{exlfig}}
\end{figure}
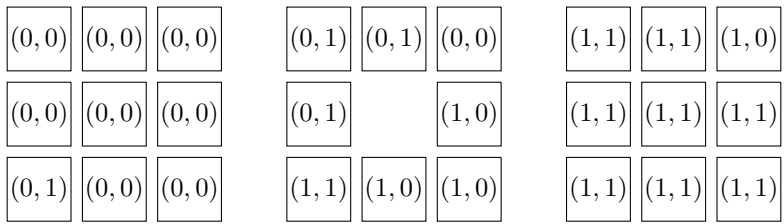
\begin{exl}
Let $B$ be the $3$-box given by $[-1,1]_\Z^3$. Then 
$\pd B = B - \{(0,0)\}$. Let $f:\pd B \to \Z^2$ be given as follows: $f(-1,-1,-1) = (0,1)$; $f(x,y,-1) = (0,0)$ for all other $x,y$; $f(-1,-1,0) = (1,1)$; $f(1,1,0) = (0,0)$; $f(x,y,0) = (0,1)$ if $y>x$ and $f(x,y,0) = (1,0)$ if $y<x$; $f(1,1,1) = (1,0)$, and $f(x,y,1)=(1,1)$ for all other $x,y$. See Figure \ref{exlfig}.

It is easy to check that $f$ is $(c_1, c_1)$-continuous, but there is no point $x$ with $f(x)\adjeq_{c_1} f(-x)$. Note that $f$ is not $(c_2,c_1)$- or $(c_3,c_1)$-continuous, since $(1,0,0)$ and $(0,1,0)$ are $c_2$- and $c_3$-adjacent, but $f(1,0,0) = (1,0)$ and $f(0,1,0) = (0,1)$ are not $c_1$-adjacent. 
\end{exl}

As we have seen in Section \ref{exlsection}, for many important examples the function $f$ will not be continuous, but rather have some Lipschitz constant $m$. Equivalently, $f$ will be $(c_l,c_k^m)$-continuous for some $m>1$. This case presents new obstacles, and there is considerable opportunity for further work. In particular we pose the following question:
\begin{question}
Given an $n$-box $B$ and a function $f:\pd B \to \Z^{n-1}$ which is $(c_l,c_k^m)$-continuous for positive integers $m$ and $k<n$ and $l\le n$, what constant $C = C(l,k,m)$ can be chosen which guarantees a point $x\in \pd B$ with $f(x) \adjeq_{c_k^C} f(-x)$?
\end{question}
We have answered the above question only in some restricted cases. By Theorem \ref{BUdim1}, in the case $n=2$, $k=1$, and $l\in \{1,2\}$ we can take $C = 2m-1$. By Theorem \ref{BUcnc1m}, for any $n$ in the case $l=n$, $k=1$ we can take $C=2m$. We believe these values of $C$ (especially the latter) may be able to be improved (reduced), which is also an interesting question.


\begin{thebibliography}{1}

\bibitem{ako08}
A.~Arslan, I.~Karaca, and A.~{\"O}ztel.
\newblock Homology groups of $n$-dimensional digital images.
\newblock {\em XXI, Turkish National Mathematics Symposium}, pages B1--13,
  2008.

\bibitem{boxe94}
L.~Boxer.
\newblock Digitally continuous functions.
\newblock {\em Pattern Recognition Letters}, 15:833--839, 1994.

\bibitem{boxe06}
L.~Boxer.
\newblock Homotopy properties of sphere-like digital images.
\newblock {\em Journal of Mathematical Imaging and Vision}, 24:167--175, 2006.

\bibitem{boxe10}
L.~Boxer.
\newblock Continuous maps on digital simple closed curves.
\newblock {\em Applied Mathematics}, 1:377--386, 2010.

\bibitem{hmps15}
J.~Haarmann, M.~Murphy, C.~Peters, and P.~C. Staecker.
\newblock Homotopy equivalence of finite digital images.
\newblock {\em To appear in Journal of Mathematical Imaging and Vision}, 2015.
\newblock arxiv eprint 1408.2584.

\bibitem{han05}
S.~E. Han.
\newblock Non-product property of the digital fundamental group.
\newblock {\em Information Sciences}, 171:73--91, 2005.

\bibitem{mato03}
J.~Matou\v{s}ek.
\newblock {\em Using the Borsuk-Ulam Theorem}.
\newblock Springer, 2003.

\bibitem{rose86}
A.~Rosenfeld.
\newblock `{C}ontinuous' functions on digital pictures.
\newblock {\em Pattern Recognition Letters}, 4:177--184, 1986.

\end{thebibliography}

\end{document}